\documentclass[12 pt]{amsart}

\usepackage{amsmath}
\usepackage{amsfonts}
\usepackage{amssymb}
\usepackage{graphicx}
\usepackage[margin=3cm]{geometry}
\usepackage{hyperref}

\usepackage{amsthm}

\newcommand{\R}{\mathbb{R}}
\newcommand{\Q}{\mathbb{Q}}
\newcommand{\Z}{\mathbb{Z}}

\DeclareMathOperator{\Diff}{Diff}
\DeclareMathOperator{\Homeo}{Homeo}

\newtheorem{theorem}{Theorem}
\newtheorem{lemma}[theorem]{Lemma}

\theoremstyle{definition}
\newtheorem{remark}[theorem]{Remark}

\newenvironment{example}[1][Example]{\begin{trivlist}
\item[\hskip \labelsep {\bfseries #1}]}{\end{trivlist}}

\begin{document}

\title[Families with strictly monotone rotation number]{One-parameter families of circle diffeomorphisms with strictly monotone rotation number}
\author{Kiran Parkhe}

\begin{abstract}
We show that if $f \colon S^1 \times S^1 \to S^1 \times S^1$ is $C^2$, with $f(x, t) = (f_t(x), t)$, and the rotation number of $f_t$ is equal to $t$ for all $t \in S^1$, then $f$ is topologically conjugate to the linear Dehn twist of the torus $\left( \begin{smallmatrix} 1&1\\ 0&1 \end{smallmatrix} \right)$.  We prove a differentiability result where the assumption that the rotation number of $f_t$ is $t$ is weakened to say that the rotation number is strictly monotone in $t$.
\end{abstract}

\maketitle

\section{Introduction}

In this article we consider homeomorphisms and diffeomorphisms of the circle $S^1 = \R/\Z$.  They are always assumed to preserve orientation.  One of the most important concepts in the classification of circle homeomorphisms is rotation number.  Rotation number can be defined in the following way.

Let $f$ be a homeomorphism of the circle, and let $\tilde{f} \colon \R \to \R$ be a lift of $f$ (that is, $\pi(\tilde{f}(x)) = f(\pi(x))$ for all $x \in \R$, where $\pi$ is the canonical projection).  Define the translation number to be $\tau(\tilde{f}) = \displaystyle\lim_{n \to \infty} \frac{\tilde{f}^n(0)}{n}$; it is a fact that this limit will always exist.  The value of $\tau(\tilde{f})$ does depend on the choice of lift, but for another lift $\tilde{f}'$ we will have $\tau(\tilde{f}) - \tau(\tilde{f}') \in \Z$.  Therefore, we define the rotation number of $f$ to be $\rho(f) = \pi(\tau(\tilde{f}))$; this does not depend on the choice of lift.  We will write rotation numbers as if they were real numbers, understanding that when we write $\rho(f) = x$ we really mean $\rho(f) = x + \Z$.

There are strong connections between the rotation number and the dynamics of $f$.  The rotation number is $0$ if and only if $f$ has at least one fixed point.  More generally, $\rho(f) \in \Q$ if and only if $f$ has a periodic orbit.  Denjoy's theorem says that if $f$ is a $C^1$ diffeomorphism such that the derivative $f'$ has bounded variation (in particular, if $f$ is $C^2$), and $f$ has irrational rotation number, then $f$ is topologically conjugate to an irrational rotation.  Note that a similar statement is false in the case of rational rotation number (even under strong smoothness assumptions): having rational rotation number implies that there exists a periodic orbit, not the much stronger statement that every orbit is periodic with the same period (which is equivalent to being conjugate to a rational rotation).

Rotation number is a function $\rho \colon \Homeo_+(S^1) \to \R/\Z$, where $\Homeo_+(S^1)$ is the group of orientation-preserving homeomorphisms of the circle.  Note that the function $\rho$ is not a homomorphism.  One can consider the restriction of $\rho$ to the subgroups $\Diff_+^r(S^1), 1 \leq r \leq \omega$.  We may endow these with the $C^r$ topology (for $\Homeo_+(S^1)$, $r = 0$), and study the function $\rho$ by investigating subspaces on which $\rho$ takes on certain values.  For example, $\rho^{-1}(\Q)$ contains an open and dense subset of $\Homeo_+(S^1)$, and $\rho^{-1}(\Q) \cap \Diff_+^r(S^1)$ contains an open and dense subset of $\Diff_+^r(S^1)$ for all $r \geq 1$ (see \cite{Herman2}).



One can gain some insight into the local structure of $\rho$ by considering it on 1-parameter families of homeomorphisms of the circle.  One way to do this is to let $f$ be a homeomorphism, and then define a family $f_t$ by $f_t = R_t \circ f$, where $R_t$ is rotation by $t$.  Letting $\rho(t) = \rho(f_t)$, the behavior of $\rho(t)$ is very different when its value is rational than when its value is irrational.  For any $\alpha \notin \Q$, $\rho^{-1}(\alpha)$ will be a single point.  This can be seen in the following way.  Without loss of generality, $\rho(f) = \alpha$.  It suffices to show that for arbitrarily small $t > 0$, and arbitrarily small $t < 0$, $f_t$ has periodic points.

Since $\rho(f) = \alpha \notin \Q$, either $f$ is conjugate to an irrational rotation, in which case every orbit is dense; or it is semi-conjugate to an irrational rotation, in which case there exists a unique invariant Cantor set on which the dynamics are minimal.  In the first case, let $x \in S^1$ be arbitrary; since the orbit of $x$ accumulates on $x$ on both sides, and any finite subset of the orbit depends continuously on $t$, for arbitrarily small $t$ (both positive and negative) composing by $R_t$ will ``close up'' the orbit of $x$, so that $f_t$ will have periodic points.  If $f$ is merely semi-conjugate to $R_\alpha$, let $x$ be in the invariant Cantor set; choose it to be an accumulation point from the left, if we want to find small $t > 0$ such that $f_t$ has periodic points, and similarly for $t < 0$.  Once we find such $x$, we argue as above.

On the other hand, for a generic (countable intersection of open and dense sets) $f \in \Homeo_+(S^1)$ or $\Diff_+^r(S^1)$ in the $C^r$ topology ($0 \leq r \leq \omega$), for every $\frac{p}{q} \in \Q$, $\rho^{-1}(\frac{p}{q})$ is a non-trivial interval (\cite{Herman}, p. 37).

The most famous example where $\rho^{-1}(\frac{p}{q})$ is a nontrivial interval for all $\frac{p}{q} \in \Q$, due to Arnold (as cited in Ghys \cite{Ghys}, p. 358), is the following: $f_a(x) = x + a\cdot sin(2\pi x)$, where $a$ is chosen to be small enough so that $f$ is a homeomorphism.  As above, $(f_a)_t$ is defined to be $R_t \circ f_a$.  For $\frac{p}{q} \in \Q$, the set of $(t_0, a_0)$ in the $(t, a)$-plane such that $\rho((f_{a_0})_{t_0}) = \frac{p}{q}$ will be a strange ``tongue'' shape.  These Arnold tongues are much studied and quite important.

\vspace{12pt}

In this paper, we consider a different type of 1-parameter family of circle homeomorphisms.  We make the strong, rather unusual assumption that the rotation number is a \textit{strictly} monotone function of the parameter; in particular, it has no ``plateaus'' when its value is rational, unlike in the above situations.  We draw some strong conclusions in this context.

The first result says that, in the special case where the rotation number is not just monotone, but the identity, the behavior is highly constrained.

\begin{theorem}
\label{ConjToLinear}
Suppose that $f \colon S^1 \times S^1 \to S^1 \times S^1$ is $C^2$, with $f(x, t) = (f_t(x), t)$, and the rotation number of $f_t$ is equal to $t$ for all $t \in S^1$.  Then $f$ is topologically conjugate to the linear Dehn twist of the torus $\left( \begin{smallmatrix} 1&1\\ 0&1 \end{smallmatrix} \right)$.
\end{theorem}

We will see, in the course of the proof, that when $t \in \Q$ the conjugacy between $f_t$ and $R_t$ is actually forced to be differentiable.

We will show by way of example that the result of Theorem 1 does not hold in general if we only assume that the rotation number is strictly monotone.  However, interesting properties continue to hold in that generality, which we explore in Theorem 2.

Note that $\frac{\partial f_t^q}{\partial t}(x, t_0)$ (or $\frac{\partial f_t^q}{\partial t}|_{t = t_0}(x)$ if we wish to vary $x$ and hold $t$ constant) means $pr_1(\frac{\partial f^q}{\partial t}(x, t_0))$, where $pr_1$ is projection onto the first component.

\begin{theorem}
\label{Differentiable}
Suppose that $f \colon S^1 \times [a, b] \to S^1 \times [a, b]$ is $C^2$, with $f(x, t) = (f_t(x), t)$, and the rotation number is a strictly monotone increasing function of $t$.  Let $t_0$ be such that $\rho(t_0) = \frac{p}{q} \in \Q$ (where this fraction is reduced).  Then $\rho$ is differentiable at $t_0$.  Moreover, if $\rho'(t_0) \neq 0$ then $\frac{\partial f_t^q}{\partial t}(x, t_0) \neq 0$ for all $x \in S^1$.  In that case,

$$\rho'(t_0) = \frac{1}{q \cdot \int_{S^1} \frac{1}{\frac{\partial f_t^q}{\partial t}(x, t_0)} dx}.$$
\end{theorem}

As we discuss below, this relates to classic results about the derivative of the rotation number when the value of the rotation number is irrational (see Herman, \cite{Herman2}).  It stands in sharp contrast to how rotation number behaves at rational values in families given by composing by rotation by $t$; see Matsumoto \cite{Matsumoto}.

\begin{remark}

The arguments in this paper actually work in somewhat greater generality than we have assumed in Theorems 1 and 2.  They imply the conclusion of Theorem 2 if we replace strict monotonicity of the rotation number with the following weaker condition:

($\star$)  Whenever $\rho(t_0) \in \Q$, there exists $\epsilon > 0$ such that either $\rho(t) > \rho(t_0)$ for $t \in (t_0, t_0 + \epsilon)$ and $\rho(t) < \rho(t_0)$ for $t \in (t_0 - \epsilon, t_0)$, or $\rho(t) < \rho(t_0)$ for $t \in (t_0, t_0 + \epsilon)$ and $\rho(t) > \rho(t_0)$ for $t \in (t_0 - \epsilon, t_0)$.

Similarly, in Theorem 1 instead of assuming $\rho(t) = t$ we could have assumed $\rho$ is strictly monotone and has nonzero derivative when its value is rational.  If we replace strictly monotone with ($\star$), then we get conjugacy to the map of the torus that preserves the circles $C_t$ of constant $t$, acting on $C_t$ by $R_{\rho(f_t)}$.  These generalizations of Theorem 1 rely on the fact that rotation number is differentiable when its value is rational, even in the generality of ($\star$), by Theorem 2.

\end{remark}

\section{Proof of Theorem 1}
Throughout the proof, $f$ is as in the statement of Theorem 1.  We use $g$ or $g_t$ to refer to more general homeomorphisms or families of homeomorphisms of the circle.

\begin{proof}[Proof of Theorem \ref{ConjToLinear}]

We must show that we can find conjugacies $\phi_t$ from $f_t$ to $R_t$ for all $t$ (i.e. $f_t = \phi_t^{-1} \circ R_t \circ \phi_t$) such that $\phi_t$ varies continuously with $t$.  Consider the case where $t \notin \Q$.  In that case, $f_t$ is $C^2$ and has irrational rotation number $t$, so by Denjoy's theorem $f_t$ is conjugate to $R_t$.  Moreover, since the only homeomorphisms of the circle that commute with a given irrational rotation are themselves rotations, if $\phi_t$ and $\tilde{\phi_t}$ are two conjugacies from $f_t$ to $R_t$ then they will differ by only a rotation.  If we agree once and for all that the conjugacies should send 0 to 0, then we have a well defined $\phi_t$ for each irrational $t$.

Now what we must show is that these $\phi_t$ ($t \notin \Q$) vary continuously in $t$, and that they extend continuously to conjugacies for $t \in \Q$.  We note that the assumption that $\rho(t) = t$, instead of merely being strictly monotone, becomes important in Lemma 6.

The following lemma implies that the $\phi_t$ ($t \notin \Q$) vary continuously in $t$.

\begin{lemma}
Let $g, g' \colon S^1 \to S^1$ be conjugate to irrational rotations, say $g = \phi^{-1} \circ R_\alpha \circ \phi$ and $g' = (\phi')^{-1} \circ R_{\alpha'} \circ \phi'$, where $\phi$ and $\phi'$ are chosen to send $0$ to $0$.  For any $\epsilon > 0$, there is a $\delta > 0$ such that if $|g - g'|_{C^0} < \delta$ then $|\phi - \phi'|_{C^0} < \epsilon$.
\end{lemma}

\begin{proof}
Let $\epsilon > 0$.  Choose $N$ large enough so that the points $0, \alpha, \ldots, N\alpha$ are $\epsilon$-dense in $S^1$ in the sense that the intervals of $S^1 \backslash \{0, \dots, N\alpha\}$ all have length less than $\epsilon$.  If $g'$ is sufficiently close to $g$, we will have $|\alpha - \alpha'| < \frac{\epsilon}{N}$, which implies that the intervals of $S^1 \backslash \{0, \dots, N\alpha'\}$ have length less than $2\epsilon$.

We would like to have control on how much $\phi$ and $\phi'$ disagree on the points $(g')^n(0), 0 \leq n \leq N$.  Since $\phi(g^n(0)) = n\alpha$ and $\phi'((g')^n(0)) = n\alpha'$, $$|\phi(g^n(0)) - \phi'((g')^n(0))| < \epsilon.$$ Since $\phi$ is uniformly continuous, if $g'$ is close enough to $g$ then for all $0 \leq n \leq N$, $(g')^n(0)$ will be close enough to $g^n(0)$ that $$|\phi(g^n(0)) - \phi((g')^n(0))| < \epsilon.$$ Therefore, $$|\phi((g')^n(0)) - \phi'((g')^n(0))| < 2\epsilon.$$

Now let $x \in S^1$, and let $0 \leq m, n \leq N$ be such that $((g')^m(0), (g')^n(0))$ is one of the intervals of $S^1 \backslash \{0, \dots, (g')^N(0)\}$ and $x \in [(g')^m(0), (g')^n(0)]$.  Then $m\alpha' \leq \phi'(x) \leq n\alpha'$.  Since $\phi((g')^m(0))$ is within $2\epsilon$ of $\phi'((g')^m(0))$, and similarly for $\phi((g')^n(0))$, $$|\phi(x) - \phi'(x)| < 2\epsilon + (n - m)\alpha'.$$  Recall that $(n - m)\alpha' < 2\epsilon$, so $$|\phi(x) - \phi'(x)| < 4\epsilon.$$  Since $\epsilon$ can be made arbitrarily small, we are done.
\end{proof}

Note that this reasoning also applies if $g'$ is conjugate to a rational rotation.  In that case, it says that for any $\epsilon > 0$, there exists $\delta > 0$ such that if $|g - g'|_{C^0} < \delta$ then $|\phi - \phi'|_{C^0} < \epsilon$, where $\phi'$ is any conjugacy of $g'$ to rotation that sends $0$ to $0$.

The following lemma shows that $f_0 = Id$.

\begin{lemma}
If a homeomorphism $g$ of the circle has $\rho(g) = 0$, and there exist homeomorphisms arbitrarily close to $g$ with positive rotation number, as well as homeomorphisms arbitrarily close with negative rotation number, then $g = Id$.
\end{lemma}

\begin{proof}
By assumption, $\rho(g) = 0$, so $g$ has fixed points.  The intuition of the lemma is as follows.  If $g \neq Id$, then either its graph crosses the diagonal $y = x$, or does not cross but is tangent to the diagonal at some points.

If the graph crosses the diagonal, then under slight perturbations it will still have fixed points, hence still have rotation number $0$.  If it is tangent to the diagonal but does not cross, then under small perturbations we may increase or decrease (but not both) the rotation number.  The reason is, in this situation the graph either does not go below, or does not go above, the diagonal.  Suppose it does not go below the diagonal.  Then by pushing the graph upward slightly we make the rotation number positive, but if we push it downward the rotation number will still be 0.  Similarly, if the graph does not go above the diagonal, then by applying a small perturbation we may decrease, but not increase, the rotation number.

Both of these possibilities contradict our situation: There exist homeomorphisms arbitrarily close to $g$ with positive rotation number, as well as homeomorphisms arbitrarily close with negative rotation number.  We wish to make this idea precise.

Suppose $g \neq Id$.  Let $p$ be a fixed point of $g$, and let $\tilde{p}$ be one of its lifts to $\R$.  Let $\tilde{g}$ be the lift of $g$ that fixes the lifts of fixed points of $g$.  Since $\tilde{g} \neq Id$, assume without loss of generality that some point is moved to the right under $\tilde{g}$, say by a distance $d$.  We will show that for $h$ sufficiently close to $g$, $\rho(h) \in [0, \frac{1}{2}]$.  If we had assumed that some point is moved to the left under $\tilde{g}$, our reasoning would show that for $h$ close to $g$, $\rho(h) \in [-\frac{1}{2}, 0]$.

Let $\delta$ be small enough so that $\tilde{g}(\tilde{p} + \delta) < \tilde{p} + \frac{1}{4}$.  Let $\epsilon = \min\{\delta, d, \frac{1}{4}\}$.  We claim that if $h$ is within $\epsilon$ of $g$, then $\rho(h) \in [0, \frac{1}{2}]$.

Let $\tilde{h}$ be the lift of $h$ which is within $\epsilon$ of $\tilde{g}$.  There are three possibilities: $\tilde{h}$ moves $\tilde{p}$ to the left, $\tilde{h}$ fixes $\tilde{p}$, or $\tilde{h}$ moves $\tilde{p}$ to the right.  In the second case, since $\tilde{h}$ has a fixed point, $\tau(\tilde{h}) = 0$, so $\rho(h) = 0$.

In the first case, note that there was a point $\tilde{x}$ moved to the right by $d$ under $\tilde{g}$.  Since $\epsilon \leq d$ and $\tilde{h}$ differs from $\tilde{g}$ by less than $\epsilon$, this point still gets moved to the right.  Since one point gets moved to the left and another gets moved to the right, some point between them gets fixed, so again $\rho(h) = 0$.

In the third case, we have $\tilde{h}(\tilde{p}) < \tilde{p} + \epsilon \leq \tilde{p} + \delta$, so $\tilde{g}(\tilde{h}(\tilde{p})) < \tilde{p} + \frac{1}{4}$, and so \begin{equation} \tilde{h}(\tilde{h}(\tilde{p})) < \tilde{p} + \frac{1}{4} + \epsilon \leq \tilde{p} + \frac{1}{2}. \end{equation} Since $$\tilde{h}^2(\tilde{p} + 1) = \tilde{h}^2(\tilde{p}) + 1 < \tilde{p} + \frac{3}{2}$$ by (1), and $\tilde{h}^2(\tilde{p}) < \tilde{p} + 1$, we have $\tilde{h}^4(\tilde{p}) < \tilde{p} + \frac{3}{2}$.  Continuing by induction, for any integer $n > 0$, $\tilde{h}^{2n}(\tilde{p}) < n - \frac{1}{2}$.  It follows that $\tau(\tilde{h}) \leq \frac{1}{2}$.  Certainly $\tau(\tilde{h}) \geq 0$, since $\tilde{h}$ sends a point (namely $\tilde{p}$) to the right.  So $\tau(\tilde{h}) \in [0, \frac{1}{2}]$ and hence $\rho(h) \in [0, \frac{1}{2}]$.

Since there are homeomorphisms arbitrarily close to $g$ with a small positive rotation number and homeomorphisms arbitrarily close with a small negative rotation number, $g = Id$.
\end{proof}

The following lemma implies that $\frac{\partial f_t}{\partial t}(x, t_0)$ is greater than $0$ for all $x$, since $\rho'(0) = 1$.

\begin{lemma}
Suppose $g \colon S^1 \times [a, b] \to S^1 \times [a, b]$ is $C^2$, with $g(x, t) = (g_t(x), t)$, and the rotation number is a strictly monotone increasing function of $t$, with $\rho(t_0) = 0$.  Suppose it is not the case that the derivative $\rho'(t_0)$ exists and is equal to $0$.  Then $\frac{\partial g_t}{\partial t}(x, t_0)$ is greater than $0$ for all $x$.
\end{lemma}

\begin{proof}
We must show that the function $\frac{\partial g_t}{\partial t}|_{t = t_0}(x)$ is strictly greater than $0$.  By Lemma 5, $g_{t_0} = Id$.  For any $t > t_0$, $g_t(x) > x$ for all $x$, and similarly for $t < t_0$; this must be true because $\rho$ is strictly monotone increasing.  Therefore, for any $x$, $\frac{\partial g_t}{\partial t}|_{t = t_0}(x) \geq 0$.  It remains to show that it never vanishes.

Suppose that for some $x_0$, $\frac{\partial g_t}{\partial t}|_{t = t_0}(x_0) = 0$.  Since the function $\frac{\partial g_t}{\partial t}|_{t = t_0}$ is nonnegative and $C^1$ (because $g$ is $C^2$), it will have a minimum at $x_0$.  Therefore, for any $\epsilon > 0$ there exists $\delta > 0$ such that when $|x - x_0| < \delta$, $\frac{\partial g_t}{\partial t}|_{t = t_0}(x) < \epsilon |x - x_0|$.

Since it is not the case that $\rho'(t_0)$ exists and is equal to $0$, and $\rho$ is monotone increasing, there exists $\eta > 0$ and $t$ arbitrarily close to $t_0$ such that $\frac{\rho(t) - \rho(t_0)}{t - t_0} > \eta$.  Without loss of generality there are $t > t_0$ arbitrarily close to $t_0$ with this property.  We may rewrite it as $\rho(t) > \eta (t - t_0)$.  Let $Q = \frac{2}{\eta} + 3(t - t_0)$, and choose a $\delta$ that works for $\epsilon = \frac{1}{Q}$. For $t$ sufficiently close to $t_0$, we have that for all $x$, $$g_t(x) - x \in ((\frac{\partial g_t}{\partial t}|_{t = t_0}(x) - \frac{\delta}{Q})(t - t_0), (\frac{\partial g_t}{\partial t}|_{t = t_0}(x) + \frac{\delta}{Q})(t - t_0)).$$  Fix a $t > t_0$ which is this close to $t_0$, and which has the property that $\rho(t) > \eta (t - t_0)$.

Consider what iterates of $g_t$ do to $x_0$.  We have $$g_t(x_0) - x_0 < \frac{\delta}{Q}(t - t_0).$$  Since this is less than $\delta$, $$g_t^2(x_0) - x_0 < (\frac{\delta}{Q} + \frac{2\delta}{Q})(t - t_0) = \frac{3\delta}{Q}(t - t_0).$$  This pattern will continue; in general, we will have $$g_t^n(x_0) - x_0 < \frac{2n - 1}{Q}\delta (t - t_0),$$ as long as $\frac{2n - 1}{Q}\delta t < \delta$.  This says that when $\frac{2n - 1}{Q}\delta (t - t_0) < \delta$, i.e. when $n < \frac{Q}{2(t - t_0)} + \frac{1}{2}$, or (substituting in the value $Q = \frac{2}{\eta} + 3(t - t_0)$) when $n < \frac{1}{\eta (t - t_0)} + 2$, \begin{equation} g_t^n(x_0) - x_0 < \delta. \end{equation}

Now notice that since $\rho(t) > \eta (t - t_0)$, there is an integer $n \leq \frac{1}{\eta (t - t_0)} + 1$ such that, after applying $g_t$ $n$ times starting at $x_0$, we will have made a full circle.  But equation (2) says that after applying $g_t$ this many times we will not have moved farther than $\delta$; this is a contradiction.
\end{proof}

\begin{lemma}
As $t \to 0$ ($t \notin \Q$), the conjugacies $\phi_t$ approach the function $x \mapsto \displaystyle\int_0^x \frac{1}{\frac{\partial f_t}{\partial t}(s, 0)} ds$.  Therefore, the conjugacies extend continuously to $t = 0$.
\end{lemma}

\begin{proof}
We must show that for every $x$, $\phi_t(x) \to \displaystyle\int_0^x \frac{1}{\frac{\partial f_t}{\partial t}(s, 0)} ds$ as $t \to 0$.

Choose $t$ small enough so that for any $x \in S^1$, $$\int_x^{f_t(x)} \frac{1}{\frac{\partial f_t}{\partial t}(s, 0)} ds \in ((1 - \epsilon)t, (1 + \epsilon)t).$$  Let $n$ be such that $nt \leq \phi_t(x) \leq (n + 1)t$.  Note that $$\int_0^x \frac{1}{\frac{\partial f_t}{\partial t}(s, 0)} ds \approx \int_0^{f_t^n(0)} \frac{1}{\frac{\partial f_t}{\partial t}(s, 0)} ds,$$ and the latter integral can be broken into integrals of the form $\displaystyle\int_{f_t^i(0)}^{f_t^{i + 1}(0)} \frac{1}{\frac{\partial f_t}{\partial t}(s, 0)} ds$, each of which lies in $((1 - \epsilon)t, (1 + \epsilon)t)$.  In this way, we find that $$\frac{1}{(1 + \epsilon)t} \int_0^x \frac{1}{\frac{\partial f_t}{\partial t}(s, 0)} ds - 1 < n < \frac{1}{(1 - \epsilon)t} \int_0^x \frac{1}{\frac{\partial f_t}{\partial t}(s, 0)} ds.$$

Now note that $\phi_t(x)$ lies between $nt$ and $(n + 1)t$, so it differs from $nt$ by at most $t$.  And $nt$ satisfies $$\frac{1}{1 + \epsilon} \int_0^x \frac{1}{\frac{\partial f_t}{\partial t}(s, 0)} ds - t < nt < \frac{1}{1 - \epsilon} \int_0^x \frac{1}{\frac{\partial f_t}{\partial t}(s, 0)} ds.$$ When $t$ and $\epsilon$ go to $0$, the difference between $nt$ and $\displaystyle\int_0^x \frac{1}{\frac{\partial f_t}{\partial t}(s, 0)} ds$ and hence also the difference between $\phi_t(x)$ and $\displaystyle\int_0^x \frac{1}{\frac{\partial f_t}{\partial t}(s, 0)} ds$ goes to $0$.
\end{proof}

To finish the proof of Theorem 1, observe that there's nothing special about $t = 0$ as opposed to any other rational number. For any $t = \frac{p}{q}$ (a reduced fraction), we may consider the function $f^q$.  By Lemma 5, this is the identity on the circle $\{t = \frac{p}{q} \}$, and nearby it will look just like $f$ did near $t = 0$. In particular, we can look at $\frac{\partial f_t^q}{\partial t}|_{t = \frac{p}{q}}$; this will be strictly greater than $0$, and we will get a function which the conjugacies must approach as $t \to \frac{p}{q}$, $t \notin \Q$.  (Note that for $t \notin \Q$, the conjugacy $\phi_t$ from $f_t$ to $R_t$ is also the conjugacy from $f_t^q$ to $R_{qt}$.)

Therefore, we have a family of conjugacies $\phi_t$ defined for every $t$ which is continuous in $t$; hence, $f$ is conjugate to the map $\left( \begin{smallmatrix} 1&1\\ 0&1 \end{smallmatrix} \right)$ of the torus.
\end{proof}

We note that the above result does not hold if we merely require the rotation number to be strictly monotone in $t$.

\begin{example}
Define a family $f_t$ parametrized by $t \in S^1$ by

\begin{displaymath}
   f_t = \left\{
     \begin{array}{lr}
       id, t = 0\\
       \phi_t^{-1} \circ R_{\rho(t)} \circ \phi_t, t \neq 0
     \end{array}
   \right.
\end{displaymath}

where $\rho \colon S^1 \to S^1$ is 1-1, sends $0$ to $0$, and is ``exponentially flat'' about $0$, i.e. there exists $\epsilon > 0$ such that for all $|t| < \epsilon$, $|\rho(t)| < e^{-\frac{1}{|t|}}$; and $\phi_t$ is defined to be $$\phi_t(x) = x + \frac{1}{4\pi}sin(\frac{1}{t})sin(2\pi x)$$ for $t \neq 0$.

Due to the exponential flatness of $\rho$, the family $f_t$ is $C^\infty$.  The given conjugacies do not vary continuously as $t \to 0$, so it looks as if $f$ may not be topologically conjugate to the map $(x, t) \mapsto (R_{\rho(t)}(x), t)$, which itself is conjugate to the map $\left( \begin{smallmatrix} 1&1\\ 0&1 \end{smallmatrix} \right)$ of the torus.  Let us see why this is true: $f$ is not conjugate to $\left( \begin{smallmatrix} 1&1\\ 0&1 \end{smallmatrix} \right)$.

Let us denote by $C_t$ the circle $\{(x, t) \colon x \in S^1 \}$; we will call these ``horizontal circles.''  Suppose, by way of contradiction, that $\psi$ is a conjugacy such that $f = \psi^{-1} \circ \left( \begin{smallmatrix} 1&1\\ 0&1 \end{smallmatrix} \right) \circ \psi$.  Then there are two possibilities: $\psi$ sends $C_t$ to $C_{\rho(t)}$ or $C_{-\rho(t)}$.  The reason is the following.  The torus is foliated by horizontal circles on which $f$ acts.  The circles on which the rotation number is irrational are dense in this foliation, so where they are sent determines the conjugacy.  Let $t$ be such that $\rho(t) \notin \Q$, and consider where $(0, t)$ is sent by $\psi$.  If $\psi(0, t) \in C_u$, then the whole orbit of $(0, t)$ must be sent into $C_u$, since $\left( \begin{smallmatrix} 1&1\\ 0&1 \end{smallmatrix} \right)$ preserves horizontal circles.  Since the orbit of $(0, t)$ is dense in $C_t$, the image of $C_t$ is $C_u$.  The dynamics of $f$ on $C_t$ must be conjugate to the dynamics of $\left( \begin{smallmatrix} 1&1\\ 0&1 \end{smallmatrix} \right)$ on $C_u$.  This implies that $u = \rho(t)$ or $-\rho(t)$.

Suppose $\psi$ sends every $C_t$ to $C_{\rho(t)}$.  When $\rho(t) \notin \Q$, we know how $\psi$ maps $C_t$ to $C_{\rho(t)}$: up to a rotation, $\psi_t$ agrees with $\phi_t(x)$ (since rotations are the only maps that commute with an irrational rotation).  Define a function $k(t) = \psi_t(0)$; note that, since $\psi$ is assumed to be continuous, $k$ is continuous.  When $\rho(t) \notin \Q$, we have $\psi_t(x) = k(t) + \phi_t(x)$.  In fact, even if $\rho(t_0) \in \Q \hspace{6pt} (\rho(t_0) \neq 0)$, we can choose $t \notin \Q \to t_0$.  As $t \to t_0$, $k(t) \to k(t_0)$ and $\phi_t \to \phi_{t_0}$, so $\psi_t \to k(t_0) + \phi_{t_0} = \psi_{t_0}$.  So for all $x$ and $t \neq 0$, $$\psi_t(x) = k(t) + \phi_t(x).$$

But notice that $$\psi_t(\frac{1}{4}) = k(t) + \frac{1}{4} + \frac{1}{4\pi}sin(\frac{1}{t}).$$  Since $k(t)$ approaches a limit as $t \to 0$, and $\frac{1}{4\pi}sin(\frac{1}{t})$ does not, $\psi_t(\frac{1}{4})$ does not approach a limit as $t \to 0$.  Therefore, $\psi$ is not continuous, a contradiction.

Now suppose that $\psi$ sends $C_t$ to $C_{-\rho(t)}$.  Arguing as above, for some continuous function $k(t)$ we will have $\psi_t(x) = k(t) - \phi_t(x)$ for all $x$ and $t \neq 0$.  But as above, if we plug in $x = \frac{1}{4}$ and let $t \to 0$, this does not approach a limit, contradicting the assumption that $\psi$ is continuous.

Therefore, we have a $C^{\infty}$ family of circle diffeomorphisms parametrized by $t \in S^1$ with rotation number a strictly monotone increasing injective function $S^1 \to S^1$, conjugate at each $t$ to a rotation by $\rho(t)$ but not globally conjugate to $\left( \begin{smallmatrix} 1&1\\ 0&1 \end{smallmatrix} \right)$.
\end{example}

\section{Proof of Theorem 2}
Before we prove Theorem 2, we remark that the subject of the derivative of the rotation number in 1-parameter families has a long and rich history.  In \cite{Herman2}, Herman essentially showed that the function $\rho \colon \Diff_+^1(S^1) \to \R/\Z$ is (Fr\'{e}chet) differentiable at irrational rotations.  The derivative is the map

$$\phi \in C^1(T^1) \mapsto \int_{S^1} \phi(x) dx \in \R$$

He used this to derive a result originally due to Brunovsk\'{y} \cite{Brunov}, that if $f_t$ is a family of homeomorphisms of the circle varying in a $C^1$ way, and $f_0$ is an irrational rotation, then $\rho(t)$ is differentiable at $t = 0$, with derivative $\displaystyle\int_{S^1} \frac{\partial f_t}{\partial t} (x, 0) dx$.  As a corollary, he shows that if $f_0$ is conjugate to an irrational rotation, say $f_0 = \phi^{-1} \circ R_\alpha \circ \phi$, then $\rho(t)$ is still differentiable at $t = 0$, with derivative $$\rho'(0) = \int_{S^1} \phi'(f_0(\phi^{-1}(x))) \cdot \frac{\partial f_t}{\partial t}(\phi^{-1}(x), 0) dx$$ (his notation was slightly different).

Recently, Matsumoto \cite{Matsumoto} showed the following: Let $f$ be a real-analytic diffeomorphism of the circle and let $f_t = R_t \circ f$.  Suppose $\rho(t) = \frac{p}{q}$, and $t$ is an endpoint of a nontrivial interval on which $\rho$ is constantly equal to $\frac{p}{q}$ (recall that for a generic $f \in \Diff_+^{\omega}(S^1)$, $\rho^{-1}(\frac{p}{q})$ will be a nontrivial interval for all $\frac{p}{q} \in \Q$).  Then $$\limsup_{t' \to t} \frac{\rho(t') - \rho(t)}{t' - t} = \infty.$$

Our context rules out this kind of behavior:


\bigskip
\noindent
{\bf Theorem~\ref{Differentiable}.}
{\em   Suppose that $f \colon S^1 \times [a, b] \to S^1 \times [a, b]$ is $C^2$, with $f(x, t) = (f_t(x), t)$, and the rotation number is a strictly monotone increasing function of $t$.  Let $t_0$ be such that $\rho(t_0) = \frac{p}{q} \in \Q$ (where this fraction is reduced).  Then $\rho$ is differentiable at $t_0$.  Moreover, if $\rho'(t_0) \neq 0$ then $\frac{\partial f_t^q}{\partial t}(x, t_0) \neq 0$ for all $x \in S^1$.  In that case,

$$\rho'(t_0) = \frac{1}{q \cdot \int_{S^1} \frac{1}{\frac{\partial f_t^q}{\partial t}(x, t_0)} dx}.$$
}
 
\bigskip

\begin{proof}
We begin by supposing that $\rho(t_0) = 0$; the general case of $\rho(t_0) = \frac{p}{q}$ is not much harder.  If $\rho'(t_0) = 0$, then we are done.  Otherwise, by Lemma 6, $\frac{\partial f_t}{\partial t}(x, t_0)$ is greater than $0$ for all $x$.  We want to show that $$\lim_{t \to t_0} \frac{\rho(t)}{t - t_0} = \frac{1}{\int_{S^1} \frac{1}{\frac{\partial f_t}{\partial t}(x, t_0)} dx}.$$  We demonstrate this limit for $t$ approaching $t_0$ from above, the case of $t$ approaching $t_0$ from below being similar.

We want to define a ``rotation time'' $T_t$, meaning basically how many times we have to apply $f_t$ before we make a full circle.  Let $\tilde{f_t}$ be the lift of $f_t$ with translation number $\tau(\tilde{f_t}) \in (0, 1)$.  Define $T_t$ to be the smallest integer such that $\tilde{f_t}^{T_t}(0) \geq 1$.

To show that $$\lim_{t \downarrow t_0} \frac{\rho(t)}{t - t_0} = \frac{1}{\int_{S^1} \frac{1}{\frac{\partial f_t}{\partial t}(x, t_0)} dx},$$ it is equivalent to show that $$\lim_{t \downarrow t_0} \frac{t - t_0}{\rho(t)} = \int_{S^1} \frac{1}{\frac{\partial f_t}{\partial t}(x, t_0)} dx.$$  Notice that $T_t \in [\frac{1}{\rho(t)}, \frac{1}{\rho(t)} + 1]$, so $$\lim_{t \downarrow t_0} \frac{t - t_0}{\rho(t)} = \lim_{t \downarrow t_0} (t - t_0)T_t$$ (provided either of these limits exists), and it suffices to show that \begin{equation} \lim_{t \downarrow t_0} (t - t_0)T_t = \int_{S^1} \frac{1}{\frac{\partial f_t}{\partial t}(x, t_0)} dx. \end{equation}

We will use the notation $\frac{\partial f_t}{\partial t}|_{t = t_0}(x)$ for $\frac{\partial f_t}{\partial t}(x, t_0)$ to emphasize that $x$ is being varied.  Let us consider Riemann sums for the integral $\displaystyle\int_{S^1} \frac{1}{\frac{\partial f_t}{\partial t}|_{t = t_0}(x)} dx$.  Specifically, let $$S_N = \displaystyle\sum_{n = 0}^{N - 1} \frac{1}{\frac{\partial f_t}{\partial t}|_{t = t_0}(\frac{n}{N})} \cdot \frac{1}{N}.$$  Since $\displaystyle\frac{\partial f_t}{\partial t}|_{t = t_0}$ is strictly greater than $0$ and continuous, $\displaystyle\frac{1}{\frac{\partial f_t}{\partial t}|_{t = t_0}}$ is continuous, so $S_N \to \displaystyle\int_{S^1} \frac{1}{\frac{\partial f_t}{\partial t}|_{t = t_0}(x)} dx$ as $N \to \infty$.  It suffices to show that for sufficiently large $N$ and $t$ close enough to $t_0$, $(t - t_0)T_t$ is within $\epsilon$ of $S_N$, since (if $N$ is large enough) $S_N$ will be within $\epsilon$ of $\displaystyle\int_{S^1} \frac{1}{\frac{\partial f_t}{\partial t}|_{t = t_0}(x)} dx$.

For any $\eta > 0$, for sufficiently large $N$, if we divide the circle into $N$ equal intervals $I_0 = [0, \frac{1}{N}), \ldots, I_{N - 1} = [\frac{N - 1}{N}, 1)$, then for any $x \in I_n$, $$\frac{\partial f_t}{\partial t}|_{t = t_0}(x) \in (\frac{\partial f_t}{\partial t}|_{t = t_0}(\frac{n}{N}) - \eta, \frac{\partial f_t}{\partial t}|_{t = t_0}(\frac{n}{N}) + \eta).$$  Furthermore, for any $\theta > 0$ we can choose $t$ close enough to $t_0$ so that for all $x$, $$f_t(x) - x \in ((\frac{\partial f_t}{\partial t}|_{t = t_0}(x) - \theta)(t - t_0), (\frac{\partial f_t}{\partial t}|_{t = t_0}(x) + \theta)(t - t_0)).$$

Let us consider the orbit of $0$ under $f_t$.  We would like estimates on the number of iterates contained in each $I_n$; let us call this $(T_t)_n$.  We have $$(T_t)_n \leq \frac{1}{(\frac{\partial f_t}{\partial t}|_{t = t_0}(\frac{n}{N}) - \eta - \theta)(t - t_0)N}.$$  Further, for any $\iota > 0$ we can choose $t$ small enough so that $$\frac{1 - \iota}{(\frac{\partial f_t}{\partial t}|_{t = t_0}(\frac{n}{N}) + \eta + \theta)(t - t_0)N} \leq (T_t)_n.$$  The number $\iota$ comes in because we have to consider how far to the right of $\frac{n}{N}$ the first iterate in $I_n$ is.  Multiplying everything by $t - t_0$, we get

$$\frac{1 - \iota}{(\frac{\partial f_t}{\partial t}|_{t = t_0}(\frac{n}{N}) + \eta + \theta)N} \leq (t - t_0)(T_t)_n \leq \frac{1}{(\frac{\partial f_t}{\partial t}|_{t = t_0}(\frac{n}{N}) - \eta - \theta)N}.$$

For any $\kappa > 0$, we can therefore say that $$(t - t_0)(T_t)_n \in ((1 - \kappa)\frac{1}{\frac{\partial f_t}{\partial t}|_{t = t_0}(\frac{n}{N})} \cdot \frac{1}{N}, (1 + \kappa)\frac{1}{\frac{\partial f_t}{\partial t}|_{t = t_0}(\frac{n}{N})} \cdot \frac{1}{N}),$$ provided we made $\eta, \theta,$ and $\iota$ small enough.  This says that $$(t - t_0)T_t = (t - t_0)(T_t)_0 + \ldots + (t - t_0)(T_t)_{N - 1} \in ((1 - \kappa)S_N, (1 + \kappa)S_N),$$ and since $S_N$ is bounded above (independent of $N$), for sufficiently large $N$ and $t$ close enough to $t_0$ we have $(t - t_0)T_t \in (S_N - \epsilon, S_N + \epsilon)$.

Finally, we must address the case where $\rho(t_0) = \frac{p}{q} \neq 0$ (assume this fraction is reduced).  Then we consider the family $g_t$, where $g_t = f_t^q$ for all $t$.  Then $g_{t_0}$ has rotation number $0$.  The above reasoning applied to the family $g_t$ says that

$$\frac{d}{dt}|_{t = t_0} \rho(g_t) = \frac{1}{\int_{S^1} \frac{1}{\frac{\partial g_t}{\partial t}(x, t_0)} dx}.$$

But $\frac{\partial g_t}{\partial t}(x, t_0) = \frac{\partial f_t^q}{\partial t}(x, t_0)$, and $\frac{d}{dt}|_{t = t_0} \rho(g_t) = q \cdot \frac{d}{dt}|_{t = t_0} \rho(f_t)$, yielding the desired formula.
\end{proof}


\begin{thebibliography}{9}

\bibitem{Brunov}
Pavol Brunovsk\'{y}.
Generic properties of the rotation number of one-paramter diffeomorphisms of the circle. \emph{Czech. Math. J.}, 24(99), 1974, 74-90.

\bibitem{Ghys}
\'{E}tienne Ghys.
Groups acting on the circle.
\emph{L'Enseignement Math\'{e}matique}, Vol. 47 (2001), p. 329-407.

\bibitem{Herman2}
Michael R. Herman.
Mesure de Lebesgue et nombre de rotation.
\emph{Lecture Notes in Math.}, No. 597 Springer Verlag, 271-293 (1977).

\bibitem{Herman}
Michael R. Herman.
Sur la conjugaison diff\'{e}rentiable des diff\'{e}omorphismes du cercle \`{a} des rotations.
\emph{Publications math\'{e}matiques de l'I.H.\'{E}.S.}, tome 49 (1979), p. 5-233.

\bibitem{Matsumoto}
Shigenori Matsumoto.
Derivatives of rotation number of one parameter families of circle diffeomorphisms.
\emph{http://arxiv.org/abs/1103.2591}

\end{thebibliography}
\end{document}